\documentclass{amsart}
\usepackage{amsthm}
\usepackage[osf,theoremfont]{newpxtext} 
\usepackage[smallerops,varg,bigdelims]{newpxmath}
\usepackage[scr=rsfso]{mathalfa}
\usepackage{bm} 
\usepackage{graphicx}
\usepackage{cite}
\usepackage{tikz}
\usepackage{enumerate}

\theoremstyle{plain}
\newtheorem{thm}{Theorem}
\newtheorem*{thm*}{Theorem}
\newtheorem{lem}[thm]{Lemma}
\newtheorem{prop}[thm]{Proposition}
\newtheorem{cor}[thm]{Corollary}

\theoremstyle{definition}
\newtheorem{defn}[thm]{Definition}
\newtheorem*{defn*}{Definition}

\newtheorem{exmp}[thm]{Example}

\theoremstyle{remark}
\newtheorem*{rem}{Remark}
\newtheorem*{note}{Note}

\def\Ind{\setbox0=\hbox{$x$}\kern\wd0\hbox to 0pt{\hss$\mid$\hss}
	\lower.9\ht0\hbox to 0pt{\hss$\smile$\hss}\kern\wd0}

\def\Notind{\setbox0=\hbox{$x$}\kern\wd0\hbox to 0pt{\mathchardef
		\nn=12854\hss$\nn$\kern1.4\wd0\hss}\hbox to
	0pt{\hss$\mid$\hss}\lower.9\ht0 \hbox to 0pt{\hss$\smile$\hss}\kern\wd0}

\def\ind{\mathop{\mathpalette\Ind{}}}

\begin{document}
\title[Locally Roelcke precompact Polish groups]{Locally Roelcke precompact Polish groups}
\author{Joseph Zielinski}
\address{Department of Mathematical Sciences,
Carnegie Mellon University,
Pittsburgh, Pennsylvania 15213-3890}
\email{zielinski.math@gmail.com}
\urladdr{http://math.cmu.edu/~zielinski}
\begin{abstract}
A Polish group is said to be \emph{locally Roelcke precompact} if there is a neighborhood of the identity element that is totally bounded in the Roelcke (or lower) group uniformity. These form a subclass of the locally bounded groups, while generalizing the Roelcke precompact and locally compact Polish groups.

We characterize these groups in terms of their geometric structure as those locally bounded groups whose coarsely bounded sets are all Roelcke precompact, and in terms of their uniform structure as those groups whose completions in the Roelcke uniformity are locally compact. We also assess the conditions under which this locally compact space carries the structure of a semi-topological semigroup. 
\end{abstract}
\maketitle

\section{Introduction}
The investigation undertaken in the present note sits at the nexus of two active areas of research into the structure of topological groups. The first of these is the topological dynamics of those groups---such as the full symmetric group on $ \mathbb{N} $, the unitary group of $ \ell_{2} $, or the automorphism group of the measure algebra of a standard Lebesgue space---that, while not necessarily compact or locally compact, are totally bounded in the \emph{Roelcke uniformity}. This feature and the connection it engenders between such groups themselves and those spaces upon which they act have been studied extensively, for example in \cite{uspenskij1998roelcke,uspenskij2001compactifications,uspenskij2008subgroups,glasner2012group,tsankov2012unitary,rosendal2013global,ben2016weakly,ibarlucia2016dynamical,ben2018eberlein}. The second follows the discovery of C. Rosendal that topological groups carry an intrinsic large-scale geometry. An early theory of this structure can be found in the preprints \cite{rosendal2014largemet,rosendal2014largeaut}, which have been collected into and elaborated upon in the manuscript \cite{rosendal-book}, and has also been explored in the papers \cite{mann2017large,zielinski2016automorphism,cohen2018large,herndon2018absolute}.

So a natural question goes, what does the study of the coarse geometry of topological groups have to say about these Roelcke precompact groups? The first answer is \emph{nothing}. All Roelcke precompact groups are coarsely bounded, and therefore have trivial coarse geometry. A second answer is subject of this paper: many of the features of the Roelcke precompact groups can be seen in a wider class of groups, one that contains both the Roelcke precompact and the locally compact groups, and for these groups those features are also reflected in their large-scale geometry.

A Polish group is \emph{locally Roelcke precompact} if some open set is totally bounded in the Roelcke uniformity. All Roelcke precompact and all locally compact Polish groups are locally Roelcke precompact. So too are the isometry groups of certain unbounded ultrahomogeneous metric spaces that behave locally like those with Roelcke precompact isometry groups (e.g., the isometry group of the Urysohn space, see Theorem \ref{example schema for lrpc isometry groups}).

The central result of this paper characterizes such groups in terms of both their uniform structure and their coarse geometry.

\begin{thm*}[below as Theorems \ref{LRPC iff locally bounded and OB=RPC} and \ref{LRPC iff completion is locally compact}]
	Suppose $ G $ is a Polish group and $ X $ is its completion in the Roelcke uniformity. The following are equivalent:
	\begin{itemize}
		\item $ G $ is locally Roelcke precompact.
		\item $ G $ is locally bounded, and every coarsely bounded subset of $ G $ is a Roelcke precompact set.
		\item $ X $ is locally compact.
	\end{itemize} 
\end{thm*}

Moreover, if $ X $ is obtained as the metric completion of a metric, $ d_{\wedge} $, compatible with the Roelcke uniformity and computed from a \emph{coarsely proper} left-invariant metric, then the extension of $ d_{\wedge} $ is a proper metric on $ X $.

The paper is outlined as follows: In Section \ref{sec background on uniform and coarse} we recall the necessary background material on the four canonical uniform structures on a topological group and on the coarse geometry of topological groups. In Section \ref{sec roelcke precompact sets} we introduce the main definitions of this paper, while Section \ref{sec examples} is devoted to examples of locally Roelcke precompact groups. Section \ref{sec ideal of roelcke precompact sets} considers the properties of the Roelcke precompact subsets of a group. Here it is seen that for locally Roelcke precompact groups, such sets are closed under taking products. This is the key fact needed for the material of Section \ref{sec characterizations}, where we prove the central result above. Finally, in Section \ref{sec semigroup operation} we consider when multiplication on $ G $ can be extended to a separately continuous operation on its Roelcke completion.

\section{Bases and metrics for the four canonical uniformities and coarse structures} \label{sec background on uniform and coarse}

Let us recall here some facts about the canonical uniformities on a topological group and of the coarse structures determined by the ideal of coarsely bounded sets. Most information in this section is found in \cite{roelcke1981uniform,uspenskij2001compactifications,rosendal2013global,ben2016weakly,nicas2012coarse,rosendal-book}. We also note some facts about bases and metrics for these coarse structures that are dual to the related results about uniformities, and are immediate, but not mentioned elsewhere. This will be used to describe the structure of the Roelcke completion of a locally Roelcke precompact Polish group in Section \ref{sec characterizations}.

\subsection{The four uniformities on a group determined by a compatible topology}
The left uniformity of a topological group is generated by the entourages $ \{(f,g) \in G^{2} \mid f \in gV \} $ as $ V $ varies over the neighborhoods of $ 1_{G} $ in $ G $. That these sets indeed form a basis for this uniformity follows from the compatibility of the topology with the group operations: that the identity neighborhoods are closed under inverses ($ V $ a neighborhood implies $ V^{-1} $ a neighborhood) and square roots ($ V $ a neighborhood implies $ W^{2} \subseteq V $ for some neighborhood $ W $). The right uniformity similarly has a basis of entourages of the form $ \{(f,g) \in G^{2} \mid f \in Vg \} $.

One uniform structure on a set is said to be \emph{finer} than those it contains---as a family of subsets of $ G^{2} $---and \emph{coarser} than those that contain it, and this partial order of containment determines a lattice structure. The join of the left and right uniformities, called the \emph{two-sided uniformity}, is generated by sets of the form $ \{(f,g) \in G^{2} \mid f \in (gV \cap Vg ) \} $. This is the coarsest uniform structure finer than both the left and right uniformities. The meet of the two (the finest uniformity coarser than both), called the \emph{Roelcke uniformity}, is generated by sets of the form $ \{(f,g) \in G^{2} \mid f \in VgV \} $. That in each case above the generating sets form bases for their respective uniformities follows again from the compatibility of the topology and the group operations. Moreover, the topology determined by each of the four uniformities agrees with the original topology on $ G $.

In the case where the topology of $ G $ is metrizable, each one of the above uniform structures has a countable base and is therefore also metrizable. Specifically, $ G $ has a left-invariant metric, $ d $, and all left-invariant metrics induce the left uniform structure. Fixing such a $ d $, compatible metrics for the right, two-sided, and Roelcke uniformities can be computed as $ d(f^{-1},g^{-1}) $, $ d(f,g) + d(f^{-1},g^{-1}) $, and $ \inf_{h \in G} \max \{d(f,h), d(h^{-1},g^{-1}) \} $, respectively. As the groups considered in what follows are all Polish, they are---in particular---metrizable, and so uniform notions (like completion or total boundedness) can be substituted with the corresponding metric notions if one prefers, and we will utilize this whenever it simplifies the exposition.

\subsection{The four coarse structures on a group determined by a compatible ideal}The situation for coarse structures on a group $ G $ is entirely dual to the above. Suppose a group, $ G $, is equipped with an ideal---viewed as the bounded subsets of the group---\emph{compatible} with the group operations in the sense that if $ A $ and $ B $ are bounded, then so are $ A^{-1} $ and the product, $ AB $. Then the entourages of the diagonal given by left translation $ \{(f,g) \in G^{2} \mid f \in gA \} $ generate a coarse structure on $ G $ whose bounded sets are the original ideal, and because the ideal and the group operations are compatible, these generating sets are a basis \cite{nicas2012coarse}.

The key insight of \cite{rosendal-book} is in identifying a canonical, compatible ideal of bounded sets possessed by every topological group.

\begin{defn*}[Rosendal \cite{rosendal-book}]
	A subset, $ A $, of a topological group, $ G $, is \emph{coarsely bounded} if it is assigned finite diameter by every continuous left-invariant pseudometric on $ G $.
\end{defn*}

These coarsely bounded sets and their associated coarse structure, the \emph{left-coarse structure}, both possess equivalent and natural intrinsic definitions, and have a nascent theory directly generalizing the geometric group theory of countable discrete (or more generally, locally compact) groups with coarse structure determined by the ideal of finite (respectively, compact) sets.

Akin to the situation for uniform structures, a coarse structure is the bounded coarse structure associated to some metric if and only if it has a countable basis. Therefore, the left-coarse structure on a group $ G $ is metrizable (and therefore by some left-invariant metric) if and only if there is a countable, cofinal family of coarsely bounded sets. In a Polish group, the Baire category theorem then implies that there is a coarsely bounded identity neighborhood, and hence a $ \mathbb{Z} $-chain of coarsely bounded open sets whose union is $ G $ and whose intersection is $ \{1_{G}\} $. Applying the metrization theorems to the corresponding chain of entourages results in a left-invariant metric, compatible with the topology of $ G $, whose bounded sets are precisely the coarsely bounded sets \cite[Theorem 2.28]{rosendal-book}. Such groups are called \emph{locally bounded}, and such a metric, compatible with both the left uniformity and the left-coarse structure, is called \emph{coarsely proper}.

Now the right-coarse structure can be defined analogously, generated by entourages of the form $ \{(f,g) \in G^{2} \mid f \in Ag \} $ and compatible with the right-invariant metric $ d(f^{-1},g^{-1}) $ if $ d $ is a left-invariant coarsely proper metric. Some attention is paid to this structure in \cite[Chapter 3 Section 4]{rosendal-book}, and particular to those groups, dual to the SIN groups, for which the left- and right-coarse structures coincide. As with the uniform structures, the coarse structures on a set also form a lattice under containment, where one is considered \emph{finer} than the another if it is contained in it \cite{dikranjan2017some}. Let us here use the convention that the lattice order is reverse containment. Then in direct analogy to the situation with the uniform structures, the join of the left- and right-coarse structure the \emph{coarsest} coarse structure \emph{finer} than both, while the meet is the \emph{finest} coarse structure \emph{coarser} than both. We collect here several facts about these coarse structures that are immediate, but have not appeared elsewhere.

\begin{prop}\label{meet and join of coarse}
	The join of the left- and right-coarse structures has as a basis the sets of the form $ \{(f,g) \in G^{2} \mid f \in (gA \cap Ag) \} $ as $ A $ varies over the coarsely bounded sets, while the meet has the sets $ \{(f,g) \in G^{2} \mid f \in AgA \} $ as a basis. The bounded sets in all four coarse structures coincide with the coarsely bounded sets. Moreover, if $ d_{L} $ is a compatible, left-invariant, coarsely proper metric on $ G $, then $ d_{\vee}(f,g) = d_{L}(f,g) + d_{L}(f^{-1},g^{-1}) $ induces the two-sided (join) uniformity and the join-coarse structure, while $ d_{\wedge}(f,g) = \inf_{h \in G}\max\{d_{L}(f,h),d_{L}(h^{-1},g^{-1})\} $ induces the Roelcke (meet) uniformity and the meet-coarse structure.
\end{prop}


\begin{proof}
	Let $ \mathcal{E}_{L} $ and $ \mathcal{E}_{R} $ denote the left- and right-coarse structures, respectively, and for $ A \subseteq G $, let 
	\begin{align*}
		E_{A}^{L} &= \{(f,g) \in G^{2} \mid f \in gA \} \\
		E_{A}^{R} &= \{(f,g) \in G^{2} \mid f \in Ag \} \\
		E_{A}^{\vee} &= \{(f,g) \in G^{2} \mid f \in (gA \cap Ag) \} \\
		E_{A}^{\wedge} &= \{(f,g) \in G^{2} \mid f \in AgA \}.
	\end{align*}
	The join, $ \mathcal{E}_{\vee}, $ is $ \mathcal{E}_{L} \vee \mathcal{E}_{R} = \mathcal{E}_{L} \cap \mathcal{E}_{R} $ (see, e.g. \cite{dikranjan2017some}, recalling our convention of reversing the lattice order). Therefore $ E \in \mathcal{E}_{\vee} $ if and only if there are coarsely bounded $ A, B \subseteq G $ with $ E \subseteq E_{A}^{L} \cap E_{B}^{R} $, that is, for all $ (f,g) \in E $, $ f \in gA \cap Bg $. Then $ f \in g(A \cup B) \cap (A \cup B)g $, so $ (f,g) \in E_{A \cup B}^{\vee} $. Meanwhile, a quick calculation shows that $ \Delta_{G} = E_{1_{G}}^{\vee} $, $ E_{A}^{\vee} \cup E_{B}^{\vee} \subseteq E_{A \cup B}^{\vee} $, $ (E_{A}^{\vee})^{-1} = E_{A^{-1}}^{\vee} $, and $ E_{A}^{\vee} \circ E_{B}^{\vee} \subseteq E_{AB \cup BA}^{\vee} $. So the sets $ E_{A}^{\vee} $ form a basis for $ \mathcal{E}_{\vee} $ as $ A $ ranges over the coarsely bounded sets.
	
	The meet, $ \mathcal{E}_{\wedge} = \mathcal{E}_{L} \wedge \mathcal{E}_{R} $ is generated by sets of the form
	\[ E_{A}^{L} \circ E_{B}^{R} \circ E_{A}^{L} \circ E_{B}^{R} \circ \dots \circ E_{A}^{L} \circ E_{B}^{R}. \]
	First note that $ (f,g) \in E_{A}^{L} \circ E_{B}^{R} $ if and only if $ f \in hA $ and $ h \in Bg $ for some $ h \in G $, if and only if $ f \in BgA $. Therefore, if $ F = \{(f,g) \in G^2 \mid f \in B^{m}gA^{m} \} $, then $ (f,g) \in F \circ E_{A}^{L} \circ E_{B}^{R} $ if and only if there is $ h \in G $ so that $ f \in B^{m}hA^{m} $ and $ h \in BgA $ if and only if $ f \in B^{m+1}gA^{m+1} $. So by induction, a composition with $ n $ blocks of $ E_{A}^{L} \circ E_{B}^{R} $ is contained in $ E_{A^{n} \cup B^{n}}^{\wedge} $. Another calculation shows that $ \Delta_{G} = E_{1_{G}}^{\wedge} $, $ E_{A}^{\wedge} \cup E_{B}^{\wedge} \subseteq E_{A \cup B}^{\wedge} $, $ (E_{A}^{\wedge})^{-1} = E_{A^{-1}}^{\wedge} $, and $ E_{A}^{\wedge} \circ E_{B}^{\wedge} \subseteq E_{AB \cup BA}^{\wedge} $. So the sets $ E_{A}^{\wedge} $ form a basis for $ \mathcal{E}_{\vee} $.
	
	The sections above $ g \in G $ in the entourage corresponding to $ A $ in each of the four coarse structures are $ gA $, $ Ag $, $ (gA \cap Ag) $ and $ AgA $, and the coarsely bounded sets are an ideal closed under products, so the bounded sets in each coarse structure coincide with the coarsely bounded sets.
	
	If $ d_{L} $ is a compatible, left-invariant, coarsely proper metric, and $ A $ a coarsely bounded set, let $ r \in \mathbb{R} $ be such that $ A \subseteq B_{d_{L}}(1_{G},r) $. Then if $ (f,g) \in E_{A}^{\vee} $, then  $ f \in gA $, so $ d_{L}(f,g) < r $ and $ f \in Ag $, so $ d_{L}(f^{-1},g^{-1}) < r $. Therefore $ d_{\vee}(f,g) < 2r $. On the other hand $ d_{\vee}(f,g) < r $ implies $ (f,g) \in E_{B_{d_{L}}(1_{G},r)}^{\vee} $.
	Similarly, if $ (f,g) \in E_{A}^{\vee} $, there are $ a,b \in A $ so that $ f = agb $, Then $ f \in agA $ and $ ag \in Ag $, so
	\[ d_{\wedge}(f,g) = \inf_{h \in G}\max\{d_{L}(f,h),d_{L}(h^{-1},g^{-1})\} \leqslant \max\{d_{L}(f,ag),d_{L}((ag)^{-1},g^{-1})\} < r. \]
	And if $ d_{\wedge}(f,g) < r $ there is $ h \in G $ so that $ d_{L}(f,h) < r $ and $ d_{L}(h^{-1},g^{-1}) $, so then $ f \in h(B_{d_{L}}(1_{G},r)) $ and $ h \in (B_{d_{L}}(1_{G},r))g $, so $ f \in (B_{d_{L}}(1_{G},r))g(B_{d_{L}}(1_{G},r))  $.
	
	That the above metrics are compatible with the two-sided and Roelcke uniformities is well-known.
\end{proof}

Note that the above proof goes through just as well when the coarsely bounded sets are replaced by any other ideal that contains all the singletons and is stable under inverses and products. Or, in a closer analogy to the correspondence between group topologies and group uniformities, a single ideal of bounded sets can be replaced with an assignment $ g \mapsto \mathcal{A}_{g} $ of ideals to points. (This generalizes $ \infty $-metric spaces, where some points may be at infinite distance from each other.)

\section{Roelcke precompact sets} \label{sec roelcke precompact sets}
Here we consider another ideal in a Polish group, that of the Roelcke precompact sets. This ideal need not be compatible in the sense of the previous section, though we will see later that it will be in several important circumstances.

\begin{defn} \label{rpc def}
	A subset, $ A \subseteq G $, of a Polish group is \emph{Roelcke precompact} if for every neighborhood, $ V \subseteq G $, of the identity there is a finite set $ F \subseteq G $ so that $ A \subseteq VFV $.
\end{defn}

\begin{rem}
	Every Polish group has a countable basis at the identity of symmetric open neighborhoods. Therefore, the above definition is equivalent to one where the $ V $ ranges only over a family of neighborhoods containing such a basis.
\end{rem}

There is another seemingly stronger, but in fact equivalent, formulation of Definition \ref{rpc def}.

\begin{prop} \label{finite subset of the RPC set}
	The above Definition \ref{rpc def} is equivalent to one that requires $ F $ to be a subset of $ A $.
\end{prop}

\begin{proof}
	Suppose $ V $ is a neighborhood of $ 1_{G} $. Let $ W $ be a symmetric neighborhood of $ 1_{G} $ with $ W^{2} \subseteq V $ and $ F' \subseteq G $ with $ A \subseteq WF'W $. Let $ F'' = \{f \in F' \mid A \cap WfW \neq \emptyset \} $, and observe that $ A \subseteq WF''W $. For each $ f \in F'' $, choose an $ a_{f} \in A $ and $ v_{f}, w_{f} \in W $ so that $ a_{f} = v_{f}fw_{f} $. Then set $ F = \{a_{f} \in A \mid f \in F''\} $.
	
	Thus if $ b \in A $, there is an $ f \in F'' $ so that 
	\[ b \in WfW = W v_{f}^{-1} a_{f} w_{f}^{-1} W \subseteq W^{2} F W^{2} \subseteq VFV. \]
\end{proof}

Thus, Definition \ref{rpc def} says precisely that $ A $ is a precompact set in $ G $ with respect to the \emph{Roelcke uniformity}, i.e., that its closure in the completion of $ G $ is compact \cite[Proposition 9.4]{roelcke1981uniform}.

In a Polish group, $ G $, a subset $ A \subseteq G $ is coarsely bounded if and only if for every identity neighborhood, $ V $, there is a finite $ F \subseteq G $ and $ n \in \mathbb{N} $ so that $ A \subseteq (FV)^{n} $ \cite{rosendal-book}. Therefore every Roelcke precompact set is coarsely bounded. Recall that a group is \emph{locally bounded} when there is a coarsely bounded identity neighborhood, and \emph{coarsely bounded} (as a group) if every subset is coarsely bounded. So the following definitions describe classes of groups that may be considered as special cases of these:

\begin{defn} \label{lrpc def}
	A Polish group, $ G $, is \emph{locally Roelcke precompact} if there is an open Roelcke precompact subset $ U \subseteq G $.
\end{defn}

\begin{defn} \label{rpc group def}
	A Polish group, $ G $, is \emph{Roelcke precompact} if $ G $ is a Roelcke precompact subset of itself.
\end{defn}

\begin{note}
	Definition \ref{rpc group def} coincides with the established usage of the term. Take care that a subgroup, $ H < G $, may be a Roelcke precompact subset of $ G $ without being a Roelcke precompact group itself. On the other hand, a Roelcke precompact group is a Roelcke precompact subset of any group in which it is continuously embedded.
\end{note}

\section{Examples of locally Roelcke precompact Polish groups} \label{sec examples}
Clearly every Roelcke precompact group is locally Roelcke precompact. Moreover, every compact subset $ A \subseteq G $ is Roelcke precompact (cover $ A $ with the open sets $ VgV $ for $ g \in A $ and take a finite subcover), and so every locally compact group is locally Roelcke precompact.

We can also describe the countable homogeneous structures, $ \mathcal{M} $, for which $ \operatorname{Aut}(\mathcal{M}) $ is locally Roelcke precompact. Recall that in such a group, the stabilizers, $ \operatorname{Stab}(\bar{a}) $, are clopen subgroups and form a basis at the identity. If $ V = \operatorname{Stab}(\bar{a}) $, then for $ f,g \in \operatorname{Aut}(\mathcal{M}) $, $ f \in VgV $ if and only if $ \operatorname{tp}(\bar{a}, f\bar{a}) = \operatorname{tp} (\bar{a},g \bar{a}) $. This is because $ f \in VgV $ if and only if $ fvg^{-1} \in V $ for some $ v \in V $, that is, if some element of $ \operatorname{Aut}(\mathcal{M}) $ fixes $ \bar{a} $ and moves $ g \bar{a} $ to $ f \bar{a} $. The homogeneity of $ \mathcal{M} $ simply allows this to be translated into the language of types.

Thus $ \operatorname{Aut}(\mathcal{M}) $ is Roelcke precompact if and only if for every $ n $-type, $ p $, there are only finitely-many $ 2n $-types that project to $ p $ in the first and last $ n $ variables. From there we see two possibilities: in the first case $ \mathcal{M} $ has finitely-many $ 1 $-types, hence the above constrains it to have finitely many $ n $-types for all $ n $, $ \operatorname{Aut}(\mathcal{M}) $ acts oligomorphically, and the theory of $ \mathcal{M} $ is $ \omega $-categorical by the Engeler-Ryll--Nardzewski-Svenonius theorem. In the second case $ \mathcal{M} $ has infinitely many $ n $-types for each $ n $ by virtue of having infinitely many $ 1 $-types. The automorphism groups of such structures are the inverse limits of oligomorphic groups by \cite{tsankov2012unitary}, and are referred to as \emph{pro-oligomorphic} in \cite{ben2018eberlein}, who observe that they are still the automorphism groups of $ \omega $-categorical structures, but in a language with infinitely-many sorts.

Then the locally Roelcke precompact groups arise when this pheneomenon is localized to the stabilizer of a finite tuple. That is, for countable homogeneous $ \mathcal{M} $, $ \operatorname{Aut}(\mathcal{M}) $ is locally Roelcke precompact if and only if there is a finite subset $ B \subseteq M $ so that for all $ \bar{a} \in M^{n} $ there are $ \bar{c}_{1}, \dots, \bar{c}_{k} \models \operatorname{tp}(\bar{a}/B) $ so that for all $ \bar{d} \models \operatorname{tp}(\bar{a}/B) $,
\[ \operatorname{tp}(\overline{ad}/B) = \operatorname{tp}(\overline{ac_{i}}/B) \]
for some $ i \leqslant k $. Or equivalently, if the conditions in the previous paragraph apply to the expansion of $ \mathcal{M} $ by some finite tuple (though note that then by the Engeler-Ryll--Nardzewski-Svenonius theorem, if the resulting group is oligomorphic, then $ \operatorname{Aut}(\mathcal{M}) $ already was).

Much of this also transfers to the language of \emph{continuous} model theory, where counterparts to many statements from classical model theory hold, and for which the automorphism groups of separable structures are exactly the Polish groups. However, we next see that certain types of metric spaces (without any additional structure) already provide a wealth of examples of locally Roelcke precompact groups, and describes the Roelcke precompact subsets in such groups. The next definition should be viewed as analogous to the above condition on types.

\begin{defn} \label{pair-propinquitous}
	A metric space, $ (X,d) $ is \emph{pair-propinquitous} if for every finite metric space, $ A $, and every $ \varepsilon > 0 $ there is a $ \delta > 0 $ so that if $ i_{1}, i_{2}, j_{1}, j_{2} $ are isometric embeddings $ A \hookrightarrow X $ with
	\[ |d(i_{1}(a),i_{2}(b)) - d(j_{1}(a),j_{2}(b))| < \delta \text{ for every } a,b \in A, \]
	then there are isometric embeddings $ j_{1}', j_{2}' \colon A \hookrightarrow X $ satisfying, for all $ a,b \in A $,
	\begin{align*}
	&d(j_{1}(a),j_{2}(b)) = d(j_{1}'(a),j_{2}'(b)),\\
	&d(i_{1}(a),j_{1}'(a)) < \varepsilon, \text{ and}\\
	&d(i_{2}(a),j_{2}'(a)) < \varepsilon.
	\end{align*}
\end{defn}

The isometry group of the Urysohn sphere, $ \mathbb{U}_{1} $, was shown to be Roelcke precompact in \cite{uspenskij2008subgroups} and \cite{rosendal2009topological} and both proofs, despite their differing approaches, involve verifying that $ \mathbb{U}_{1} $ possesses a stronger property: for every $ n \in \mathbb{N} $ and every $ \varepsilon > 0 $ there is a $ \delta > 0 $ so that whenever any two enumerated $ n $-point subsets, $ b_{1}, \dots, b_{n} $ and $ c_{1}, \dots, c_{n} $ in $ \mathbb{U}_{1} $ agree on distances up to $ \delta $ (i.e., $ \max_{i,j \leqslant n} |d(b_{i},b_{j}) - d(c_{i},c_{j})| < \delta $) there is an isometric copy of one, $ c_{1}', \dots, c_{n}' $, pointwise within $ \varepsilon $ of the other ($ \max_{i \leqslant n} d(b_{i},c_{i}') < \varepsilon $). Maybe this property should be called \emph{propinquity}, but regardless, the definition here weakens it by requiring a particular $ \delta $ not to work for all metric spaces of a given size, but only those comprised of a pair of isometric copies of a fixed finite subset.

Recall that a metric space, $ (X,d) $, is \emph{ultrahomogeneous} if every isometry between finite subsets extends to an isometry of $ X $ and \emph{approximately ultrahomogeneous} if, for every $ \varepsilon > 0 $ and every isometry, $ f $, between finite subsets, there is a global isometry that agrees with $ f $ up to $ \varepsilon $. Clearly every ultrahomogeneous metric space is approximately ultrahomogeneous. If $ A \subseteq_{\mathrm{fin}} X $ and $ r > 0 $,
\[ V_{A,r} = \{\forall a \in A \; d(a,f(a)) < r \} \]
is a symmetric, open neighborhood of $ \mathrm{id}_{X} $ in $ \operatorname{Iso}(X) $, and such sets form an identity basis.

\begin{thm} \label{example schema for lrpc isometry groups}
	Suppose $ (X,d) $ is a separable, complete, pair-propinquitous, approximately ultrahomogeneous metric space. Then for every $ x \in X $ and $ r \in \mathbb{R}^{+} $, $ V_{x,r} $ is a Roelcke precompact subset of $ \operatorname{Iso}(X) $. In particular, $ \operatorname{Iso}(X) $ is locally Roelcke precompact.
\end{thm}

\begin{proof}
	Let $ W $ be an open neighborhood of $ \operatorname{Iso}(X) $. By shrinking $ W $, we may assume that $ W = V_{\{y_{0},\dots,y_{n}\},\varepsilon} $, where $ y_{0} = x $ and $ \varepsilon \leqslant r $. Let $ \delta $ be the value given by pair-propinquity with respect to the finite metric space $ A=\{y_{0},\dots,y_{n}\} $ and the value $ \frac{\varepsilon}{2} $.
	
	Observe that for any isometry $ h \in V_{x,r} $ and any $ i,j \leqslant n $,
	\[  d(y_{k},h(y_{l})) \leqslant d(y_{k},x) + d(x,h(x)) + d(h(x),h(y_{l})) < r + 2 \operatorname{diam}(A). \]
	Let $ \mathcal{P} $ be a partition of the interval $ [0,r + 2 \operatorname{diam}(A)] $ into finitely-many subintervals of length less than $ \delta $. For each $ s \colon \{0,\dots,n\}^{2} \to \mathcal{P} $, fix $ f_{s} \in \operatorname{Iso}(X) $ satisfying $ d(y_{k},f_{s}(y_{l})) \in s(k,l) $ for all $ k,l \leqslant n $, if one exists, and let $ F $ be the set of these isometries. Note that $ F \neq \emptyset $, as $ \operatorname{id}_{X} $ satsfies the conditions of $ f_{s_{\mathrm{id}}} $ for $ s_{\mathrm{id}} \colon \{0,\dots,n\}^{2} \to \mathcal{P} $, where $ s_{\mathrm{id}}(k,l) $ is the unique interval containing $ d(y_{k},y_{l}) $.
	
	Then for any $ g \in V_{x,r} $, if $ s_{g}\colon \{0,\dots,n\}^{2} \to \mathcal{P} $ satisfies $ d(y_{k},g(y_{l})) \in s_{g}(k,l) $ for all $ k,l \leqslant n $, then the values $ d(y_{k},g(y_{l})) $ and $ d(y_{k},f_{s_{g}}(y_{l})) $ lie in the same piece, $ s_{g}(k,l) $, of the partition, $ \mathcal{P} $, and so  $ \left| d(y_{k},g(y_{l})) - d(y_{k},f_{s_{g}}(y_{l})) \right| < \delta $. Thus by pair-propinquity there is an isometric copy of $ A \cup f_{s_{g}}[A] $ located point-by-point within $ \frac{\varepsilon}{2} $ of $ A \cup g[A] $ and by approximate ultrahomogeneity, there is a global isometry, $ u_{g} $, that agrees with this partial isometry on $ A \cup f_{s_{g}}[A] $ with error at most $ \frac{\varepsilon}{2} $. Thus, for all $ k \leqslant n $, $ d(g(y_{k}),u_{g}f_{s_{g}}(y_{k})) < \frac{\varepsilon}{2} + \frac{\varepsilon}{2} = \varepsilon $ and so $ d(f_{s_{g}}^{-1}u_{g}^{-1}g(y_{k}),y_{k}) = d(g(y_{k}),u_{g}f_{s_{g}}(y_{k})) < \varepsilon $ and $ f_{s_{g}}^{-1}u_{g}^{-1}g \in W $.
	
	Moreover, for all $ k \leqslant n $, $ d(y_{k},u_{g}(y_{k})) < \frac{\varepsilon}{2} $, and so $ u_{g} \in V_{A,\frac{\varepsilon}{2}} \subseteq W $. So as $ f_{s_{g}}^{-1}u_{g}^{-1}g \in W $, $ g \in u_{g}f_{s_{g}}W \subseteq Wf_{s_{g}}W \subseteq WFW $. So $ V_{x,r} \subseteq WFW $ and as $ W $ was arbitrary, $ V_{x,r} $ is a Roelcke precompact subset of $ \operatorname{Iso}(X) $.
\end{proof}

Note that if an ultrahomogeneous $ (X,d) $ has finite diameter, then $ \operatorname{Iso}(X) = V_{x,r} $ for any $ x $ and sufficiently large $ r $, and so if the other conditions of Theorem \ref{example schema for lrpc isometry groups} are met, then $ \operatorname{Iso}(X) $ is Roelcke precompact. Many well-known examples of Roelcke precompact groups can then be seen as instances of this theorem, for example $ \operatorname{Iso}(\mathbb{U}_{1}) $ as already mentioned, as well as $ S_{\infty} $ (\cite{roelcke1981uniform}, viewed as the isometry group of a countable set with the discrete metric) $ \operatorname{Aut}(\mathbf{R}) $ (the random graph, as a metric space with the graph metric), and $ U(\ell_{2}) $ (\cite{uspenskij1998roelcke}, viewed as the isometry group of the unit sphere of $ \ell_{2} $).

This also furnishes a number of examples of locally Roelcke precompact groups that are not Roelcke precompact. For example, the isometry group, $ \operatorname{Iso}(\mathbb{U}) $ of the Urysohn space, the (affine) isometry groups of $ \ell_{2} $ and the Gurarij space, and $ \operatorname{Aut}(\mathcal{T}_{\infty}) $, the automorphism group of the (unrooted) countably-regular tree. For this last example, note that $ \mathcal{T}_{\infty} $ is automatically pair-propinquitous by virtue of having a uniformly discrete set of possible distances. So the same applies to the automorphism group of any \emph{metrically homogeneous} graph---a graph that is ultrahomogeneous when viewed as a metric space with the graph metric.

As most familar examples of ultrahomogeneous metric spaces are enumerated in the examples above, it is worth observing that pair-propinquity is not a redundant condition. As the next example shows, there are ultrahomogeneous metric spaces lacking this property, and whose isometry groups are indeed not locally Roelcke precompact.

\begin{exmp}
	Let $ S \subseteq \mathbb{R} $ be $ S = \mathbb{N} \cup \{1 + \frac{1}{n} \mid n \geqslant 3 \} $. Observe that $ 3 $ points with an assignment of distances from $ S $ satisfy the triangle inequality if and only if the same is true when the values of the form $ 1 + \frac{1}{n} $ are replaced with value $ 1 $. Then as the set $ \mathbb{N} $ satisfies the \emph{4-values condition} of \cite{delhomme2007divisibility}, so does $ S $, and so by \cite{sauer2013distance} there is a countable ultrahomogeneous metric space, $ \mathbb{U}_{S} $, for which $ S = \{d(x,y) \mid x,y \in U_{S}\} $ and containing a copy of every finite metric space whose distances come from $ S $. The space $ \mathbb{U}_{S} $ fails to have pair-propinquity as witnessed by $ A = \{x\} $ (any singleton) and $ \varepsilon = \frac{1}{2} $. For any $ \delta > 0 $, let $ n > \frac{1}{\delta} $ and by universality find embeddings with $ d(i_{1}(x), i_{2}(x))=1 $ and $ d(j_{1}(x), j_{2}(x))=1 + \frac{1}{n} $. For any $ y,z \in \mathbb{U}_{S} $, $ d(y,z) < \frac{1}{2} $ implies $ y = z $, so there can be no $ j_{1}' $ and $ j_{2}' $ satisfying the conditions of Definition \ref{pair-propinquitous}.
	
	Let $ G = \operatorname{Iso}(\mathbb{U}_{S}) $; then $ G $ is not locally Roelcke precompact. For if $ U \subseteq G $ is an identity neighborhood, then $ U $ contains an open neighborhood of the form $ V_{A,\varepsilon} $, and so contains the stabilizer, $ V_{A} $, of $ A $. Fix $ K \in \mathbb{N} $ sufficiently large (at least twice $ \operatorname{diam}(A) $) and let $ y \in \mathbb{U}_{S} $ be a point for which $ d(a,y) = K $ for all $ a \in A $. Now consider the open neighborhood $ W = V_{y,\frac{1}{2}} $  ($ = \operatorname{Stab(y)} $, as $ S $ contains no positive values less than $ 1 $). Observe that if $ g \in WfW $, there are $ u,v \in W $ with $ g = ufv $ and so $ d(y,g(y)) = d(y,ufv(y)) = d(u^{-1}(y), fv(y)) = d(y,f(y)) $, as $ u^{-1} $ and $ v $ are isometries fixing $ y $. For any finite $ F \subseteq G $, there is an $ M \geq 3 $ so that $ d(y,f(y)) \neq 1 + \frac{1}{M} $ for any $ f \in F $. By universality and ultrahomogeneity of $ \mathbb{U}_{S} $, there is a $ z $ so that $ d(z,a) = K $ for all $ a \in A $ and $ d(y,z) = 1 + \frac{1}{M} $. Then by ultrahomogeneity again there is a $ g \in G $ fixing $ A $ and sending $ y \mapsto z $. So as $ g \in U $ since it fixes $ A $, but $ d(y,g(y)) = d(y,z) = 1 + \frac{1}{M}  $, so $ g \notin WFW $. So $ U \nsubseteq WFW $, and since $ F $ was an arbitrary finite subset of $ G $, $ U $ is not Roelcke precompact.
\end{exmp}

\section{The ideal of Roelcke precompact sets} \label{sec ideal of roelcke precompact sets}

\begin{lem} \label{closure properties of R}
	The family of Roelcke precompact subsets of $ G $ is closed under taking subsets, inverses, finite unions, and left and right translations.
\end{lem}

\begin{proof}
	Suppose $ A $ and $ B $ are Roelcke precompact subsets of $ G $ and that $ V $ is a neighborhood of the identity.
	\begin{itemize}
		\item If $ C \subseteq A $, then there is an $ F \subseteq_{\mathrm{fin}} G $ so that $ A \subseteq VFV $, and consequently $ C \subseteq VFV $.
		\item There is a finite $ F $ with $ A \subseteq V^{-1} F V^{-1} $, and so $ A^{-1} \subseteq V F^{-1} V $.
		\item Let $ F_{A} $ and $ F_{B} $ be finite sets with $ A \subseteq VF_{A}V $ and $ B \subseteq V F_{B} V $. Then $ A \cup B \subseteq V (F_{A} \cup F_{B}) V $.
		\item Suppose $ g \in G $ and $ W = V \cap g^{-1} V g \cap g V g^{-1} $. If $ F \subseteq_{\mathrm{fin}} G $ is such that $ A \subseteq WFW $, then
		\[ gA \subseteq gWFW \subseteq g g^{-1} V g F W = V (gF) W \subseteq V (gF) V \]
		and
		\[ Ag \subseteq WFWg \subseteq WFgVg^{-1}g = W(Fg)V \subseteq V(Fg)V. \]
	\end{itemize}
\end{proof}

For $ A,B \subseteq G $ let $ AB =\{gh \in G \mid g \in A \text{ and } h \in B \} $. In the next example, we observe that $ AB $ needn't be Roelcke precompact even when both $ A $ and $ B $ are.

\begin{exmp} \label{not stable under products exmp}
	Let $ G = \mathbb{Z}^{\mathbb{N}} \rtimes S_{\infty} $ be the semidirect product given by the action of $ S_{\infty} $ that permutes the coordinates of elements of $ \mathbb{Z}^{\mathbb{N}} $. As $ S_{\infty} $ is a Roelcke precompact \emph{group} (in the sense of Definition \ref{rpc group def}) \cite[Example 9.4]{roelcke1981uniform}, the subgroup $ A = \{1_{\mathbb{Z}^{\mathbb{N}}}\} \times S_{\infty} $ is a Roelcke precompact \emph{subset} of $ G $ (in the sense of Definition \ref{rpc def}). Letting $ g = ((1,2,3,\dots),1_{S_{\infty}}) $, the coset $ gA $ is also Roelcke precompact by Lemma \ref{closure properties of R}.
	
	However, $ A(gA) $ is not: note that for every $ m \in \mathbb{N} $, it contains an element of the form $ (x,1_{S_{\infty}}) $ where $ x_{1} = m $, specifically $ h_{m}gh_{m} $ where $ h_{m} = (1_{\mathbb{Z}^{\mathbb{N}}},(1 \quad m)) $ is the pair whose first coordinate is $ 1_{\mathbb{Z}^{\mathbb{N}}} $ and whose second coordinate is the permutation that exchanges $ 1 $ and $ m $. But if $ V = \{(x,\sigma) \in G \mid x_{1}=0 \text{ and } \sigma(1) = 1 \} $, then $ AgA \nsubseteq VFV $ for any finite $ F $, as
	 \[ \{m \in \mathbb{N} \mid \exists (y, \tau) \in VFV \; y_{1} = m  \} = \{m \in \mathbb{N} \mid \exists (y, \tau) \in F \; y_{1} = m  \}. \]
\end{exmp}

An interesting feature of the group in the above example is that it \emph{is} coarsely bounded. In fact, for every identity neighborhood, $ V $, there is an $ f \in G $ so that $ G = VfVf^{-1}V $. To see this note that the group $ G $ can be viewed as the automorphism group $ G = \operatorname{Aut}(\mathcal{M}) $, where $ \mathcal{M} = (\bigsqcup_{n \in \mathbb{N}} \mathbb{Z}, S) $ consists of countably-many copies of $ \mathbb{Z} $ equipped with a function for successor. Then $ \operatorname{Th}(\mathcal{M}) $ is $ \omega $-stable and $ \mathcal{M} $ is saturated, implying that $ G $ is coarsely bounded, but moreover the relation on finite subsets, where $ A \ind B $ holds when no element of $ A $ is in the same copy of $ \mathbb{Z} $ as an element of $ B $, is an \emph{orbital independence relation}, and so the coarse boundedness of $ G $ takes the above form \cite[Chapter 6]{rosendal-book}. In particular, $ G $ is locally bounded, and the ideals of Roelcke precompact sets and of coarsely bounded sets do not coincide.

However, in some circumstances, the ideal of Roelcke precompact sets will be stable under products.

\begin{prop} \label{R closed under products in Weil complete}
	If $ G $ is Weil complete, the product of two Roelcke precompact subsets of $ G $ is Roelcke precompact.
\end{prop}

\begin{proof}
	By \cite[Theorem 11.4]{roelcke1981uniform}, if $ G $ is Weil complete (i.e., complete in the left uniformity) then it is also complete in the Roelcke uniformity. Thus every Roelcke precompact subset is, in fact, compact and the product of two compact subsets is a compact subset.
\end{proof}

\begin{prop} \label{R closed under products in LRPC}
	If $ G $ is locally Roelcke precompact, the product of two Roelcke precompact subsets of $ G $ is Roelcke precompact.
\end{prop}

\begin{proof}
	Fix an open, Roelcke precompact subset $ U $, which we may assume contains $ 1_{G} $ by Lemma \ref{closure properties of R}. Now suppose that $ A $ and $ B $ are also Roelcke precompact subsets of $ G $, and that $ V $ is an identity neighborhood in $ G $. We will produce a finite $ F \subseteq G $ with $ AB \subseteq VFV $.
	
	Choose an open neighborhood of the identity $ W $ satisfying $ W^{2} \subseteq V \cap U $. So there are finite $ E_{A} $ and $ E_{B} $ so that $ A \subseteq WE_{A}W $ and $ B \subseteq W E_{B} W $. Recall from Lemma \ref{closure properties of R} that the collection of Roelcke precompact subsets of $ G $ is hereditary and closed under left/right translations and finite unions. Thus $ W^{2} \subseteq U $ is Roelcke precompact, as are $ W^{2} E_{B} = \bigcup_{f \in E_{B}} W^{2}f $ and $ E_{A}W^{2}E_{B} = \bigcup_{f \in E_{A}} f W^{2} E_{B} $. So there is a finite $ F \subseteq G $ with $ E_{A} W^{2} E_{B} \subseteq W F W $. Therefore,
	\[ AB \subseteq (W E_{A} W)(W E_{B} W) = W (E_{A} W^{2} E_{B}) W \subseteq W (WFW) W = W^{2} F W^{2} \subseteq VFV. \]
\end{proof}

This fact will be key in characterizing locally Roelcke precompact groups below.

\section{Equivalent characterizations of locally Roelcke precompact Polish groups} \label{sec characterizations}

In this section, we let $ G $ denote an arbitrary a Polish group, and $ X $ its completion in the Roelcke uniformity. Equivalently, $ X $ is the uniform structure associated to the metric completion of $ (G,d_{\wedge}) $, where $ d_{\wedge} $ is computed from any left-invariant metric, $ d_{L} $, as $ d_{\wedge}(f,g) = \inf_{h \in G}\max\{d_{L}(f,h),d_{L}(h^{-1},g^{-1})\} $. Since a subset of $ G $ is Roelcke precompact if and only if its closure in the Roelcke completion of $ G $ is compact and, specifically, $ G $ is a Roelcke precompact group if and only if $ X $ is compact, one may then wonder if the completion of a \emph{locally} Roelcke precompact group must be \emph{locally} compact.

Let us first observe that this is not an immediate parsing of definitions as it is for Roelcke precompact groups. One definition says that every $ x \in X $ has a compact neighborhood, while the other says this is true of every $ x $ in the comeagre subset $ G \subseteq X $. And there are certainly instances of non-locally compact spaces with comeagre locally compact subsets, for instance $ \{\lambda e_{k} \in \ell_{2} \mid k \in \mathbb{N} \text{ and } \lambda \in [0,1] \} $ becomes locally compact when $ \vec{0} $ is removed. And though $ X $ has a lot more homogeneity than that example---as multiplication extends to an action $ G \curvearrowright X $ with comeagre orbit $ G \subseteq X $---it is not a matter of taking a compact neighborhood of an element of $ G $ and pushing it around by this action, as the next example shows.

\begin{exmp} \label{example aut tree}
	Let $ G = \operatorname{Aut}(\mathcal{T}_{\infty}) $, which is locally Roelcke precompact by Theorem \ref{example schema for lrpc isometry groups}. Fix any elements $ a,b \in \mathcal{T}_{\infty} $. If $ V = \operatorname{Stab}(a,b) $, then as noted in Section \ref{sec examples}, $ f \in VgV $ if and only if $ (a,b,ga,gb) \models \operatorname{tp}(a,b,fa,fb) $, and therefore $ d(a,fb) = d(a,gb) $ in the graph metric on $ \mathcal{T}_{\infty} $. So every Roelcke-Cauchy sequence, $ (f_{n})_{n \in \mathbb{N}} $, eventually decides a value for $ d(a,f_{n}b) $ and the function $ \Delta \colon G \mapsto \mathbb{N} $ taking $ g \mapsto d(a,gb) $ is uniformly continuous with respect to the Roelcke uniformity on $ G $ and the discrete uniformity on $ \mathbb{N} $ and therefore extends to a continuous function $ \Delta \colon X \to \mathbb{N} $.
	
	So if $ K $ is any compact subset of $ X $, $ K $ sees only finitely many values in the above function, say $ R = \max \Delta[K] + 1  $. Fix any enumeration $ p_{0}, p_{1}, p_{2}, \dots $ of $ \mathcal{T}_{\infty} $, and for each $ n \in \mathbb{N} $ let $ f_{n} \in G $ be an automorphism satisfying $ d(p_{i},f_{n}p_{j}) = d(p_{i},p_{0}) + R + d(p_{0},p_{j}) $ for all $ i,j \leqslant n $. Such an automorphism can be found by fixing a $ q \in \mathcal{T}_{\infty} $ of distance $ R $ from $ p_{0} $, setting $ p_{0} \mapsto q $, and choosing the elements $ f_{n}p_{i} $ to lie on the subtree obtained by removing all vertices whose path to $ q $ passes through the path from $ p_{0} $ to $ q $. Then $ (f_{n})_{n \in \mathbb{N}} $ is Roelcke-Cauchy, and $ y = \lim_{n} f_{n} \in X $ has the following property: for any $ g \in G $,
	\[ \Delta(g \cdot y) = \lim_{n} d(g^{-1}a,f_{n}b) \geqslant R \]
	Therefore $ G \cdot y \cap K = \emptyset $, or equivalently $ y \cap G \cdot K = \emptyset $. So for any compact set in $ X $ (and so also for any open set with compact closure) there is an element of $ X $ that does not lie in any translate.
\end{exmp}

However, we will indeed show that the Roelcke completion of a locally Roelcke precompact group is locally compact. As it turns out, the example above is characteristic of groups that do not have \emph{bounded geometry}. Specifically, it is shown in \cite[Chapter 5]{rosendal-book} that in the action $ G \curvearrowright X $ of a locally Roelcke precompact group on its Roelcke completion that extends left multiplication, there is a compact set whose $ G $-saturation is $ X $ if and only if $ G $ has bounded geometry. He then obtains a characterization of bounded geometry for all Polish groups in terms of certain types of actions on locally compact spaces. This is done by embedding such a group into $ \operatorname{Iso}(\mathbb{U}) $ and analyzing the action of $ G $ on a closed, invariant subspace of the completion of $ \operatorname{Iso}(\mathbb{U}) $, which is locally compact by Theorem \ref{LRPC iff completion is locally compact} below.

Our first characterization of the locally Roelcke precompact Polish groups is geometric, identifying them among the locally bounded ones. Recall that every Roelcke precompact subset of a Polish group is coarsely bounded, and so these groups must, in particular, be locally bounded. In general, a Polish group may contain coarsely bounded subsets that are not Roelcke precompact, for example the group, $ \mathbb{Z}^{\mathbb{N}} \rtimes S_{\infty} $, of Example \ref{not stable under products exmp}. However, the locally Roelcke precompact Polish groups are precisely the locally bounded ones in which these two ideals coincide.

\begin{thm} \label{LRPC iff locally bounded and OB=RPC}
	A Polish group, $ G $, is locally Roelcke precompact if and only if it is locally bounded and every coarsely bounded subset is Roelcke precompact. 
\end{thm}

\begin{proof}
		($ \impliedby $): As $ G $ is locally bounded, there is a coarsely-bounded neighborhood, $ U $. And as every coarsely bounded set is Roelcke precompact, so too is $ U $.
		
		($ \implies $): Let $ U $ be a Roelcke precompact neighborhood. Then $ U $ is, in particular, coarsely bounded, and so $ G $ is locally bounded. If $ A \subseteq G $ is any coarsely bounded set, there is a finite $ F $ and a bound $ k \in \mathbb{N} $ so that $ A \subseteq (FU)^{k} $. By Lemma \ref{closure properties of R}, the Roelcke precompact subsets are closed under finite unions and left translations. Thus $ FU = \bigcup_{f \in F} fU $ is Roelcke precompact. Then by Proposition \ref{R closed under products in LRPC}, $ (FU)^{k} $ is Roelcke precompact, and again by Lemma \ref{closure properties of R}, so too is $ A $.
\end{proof}

In particular, if a Polish group is coarsely bounded but not Roelcke precompact, then it also fails to be locally Roelcke precompact.

\begin{cor} \label{RPC iff LRPC and OB}
	A Polish group is Roelcke precompact if and only if it is coarsely bounded and locally Roelcke precompact.
\end{cor}

\begin{proof}
	If $ G $ is locally Roelcke precompact and coarsely bounded, then by Theorem \ref{LRPC iff locally bounded and OB=RPC}, every coarsely bounded set---in particular, $ G $ itself---is Roelcke precompact.
\end{proof}

Since a locally Roelcke precompact Polish group, $ G $, is locally bounded, it admits a compatible, left-invariant, \emph{coarsely proper} metric, $ d_{L} $. Specifically, such a metric $ d_{L} $ induces both the left uniformity and the left-coarse structure, and therefore it is also compatible with the topology of $ G $ and assigns finite diameter to a subset if and only if that subset is coarsely bounded. Fix such a metric, and form the associated metric $ d_{\wedge}(f,g) = \inf_{h \in G}\max\{d_{L}(f,h),d_{L}(h^{-1},g^{-1})\} $. Then by Proposition \ref{meet and join of coarse}, $ d_{\wedge} $ is compatible with the Roelcke uniformity and the meet-coarse structure. Therefore
\begin{itemize}
	\item the metric, $ d_{\wedge} $, is compatible with the topology on $ G $,
	\item the metric completion, $ (X,d_{\wedge}) $, of $ (G,d_{\wedge}) $ is the Roelcke completion of $ G $, and
	\item the metric $ d_{\wedge} $ assigns finite diameter to a subset of $ G $ if and only if it is coarsely bounded.
\end{itemize}
Though they always agree on the bounded sets, in general the spaces $ (G,d_{L}) $ and $ (G,d_{\wedge}) $ are not coarsely equivalent. For example, $ (\operatorname{Aut}(\mathcal{T}_{\infty}),d_{\wedge}) $ is coarsely equivalent to its completion, which by the following theorem is a proper metric space. But by Rosendal's characterization of groups with bounded geometry, $ (\operatorname{Aut}(\mathcal{T}_{\infty}),d_{L}) $ cannot be coarsely equivalent to such a space.

Recall that a metric on a locally compact space is \emph{proper} if it satisfies the Heine-Borel theorem---that is, if closed, bounded subsets of the space are compact. Now we see that the locally Roelcke precompact Polish groups are precisely those for which the completion in the Roelcke uniformity is locally compact.

\begin{thm} \label{LRPC iff completion is locally compact}
	A Polish group, $ G $, is locally Roelcke precompact if and only if its completion, $ X $, in the Roelcke uniformity is locally compact. In this case, if $ d_{L} $ is a compatible, left-invariant, coarsely-proper metric for $ G $, then the extension of $ d_{\wedge} $ to $ X $ is a proper metric.
\end{thm}

\begin{proof}
	If $ X $ is locally compact, the element $ 1_{G} \in G \subseteq X $ has a compact neighborhood, $ K \subseteq X $, and so $ K \cap G $ is a Roelcke precompact neighborhood of $ 1_{G} $ in $ G $.
	
	Conversely, we suppose $ G $ is locally Roelcke precompact and that $ d_{\wedge} $ is computed from a compatible, left-invariant, coarsely proper metric, $ d_{L} $, as described above. Now suppose $ A $ is a closed, bounded subset of $ (X,d_{\wedge}) $. Then $ A \subseteq B = B_{(X,d_{\wedge})}(h,r)$ for some $ h \in G \subseteq X $ and $ r > 0 $. Then $ B \cap G = B_{(G,d_{\wedge})}(h,r) $ is a coarsely bounded subset of $ G $, and so is Roelcke precompact by Theorem \ref{LRPC iff locally bounded and OB=RPC}. Thus the closure in $ X $, $ \overline{B} = \overline{B \cap G} $, is compact, while $ A $ is a closed subset of $ \overline{B} $ and so also compact. So $ (X,d_{\wedge}) $ is a proper metric space, and in particular $ X $ is locally compact.
\end{proof}

One consequence of the above is a restriction on properties of locally Roelcke precompact groups. Recall that $ G $ is \emph{Weil complete} if it is complete in the left uniformity. For Polish groups this is equivalent to the existence of a complete, left-invariant metric, and so these groups are also called \emph{CLI} (``complete left invariant''). For example, any solvable Polish group is Weil complete \cite[Corollary 3.7]{hjorth1999vaught}, as is any locally compact Polish group. These latter groups are the only Weil complete locally Roelcke precompact Polish groups.

\begin{cor} \label{cor loc comp iff cli and lrpc}
	A Polish group is locally compact if and only if it is Weil complete and locally Roelcke precompact.
\end{cor}

\begin{proof}
	Every locally compact Polish group is complete in the left uniformity. And in general, every compact set is Roelcke precompact, so every locally compact Polish group is also locally Roelcke precompact.
	
	Conversely, if $ G $ is Weil complete and Polish, then $ G $ is also complete in the Roelcke uniformity \cite[Proposition 11.4]{roelcke1981uniform} and thus $ G $ coincides with its Roelcke completion, and so is locally compact by Theorem \ref{LRPC iff completion is locally compact}.
\end{proof}

\subsection{Closure properties}

\begin{thm}
	Suppose $ G $ is Polish and $ H $ is an open subgroup of $ G $. Then $ G $ is locally Roelcke precompact if and only if $ H $ is.
\end{thm}

\begin{proof}
	Suppose $ U \subseteq G $ is an open Roelcke precompact identity neighborhood. By Lemma \ref{closure properties of R}, $ W = U \cap H $ is a Roelcke precompact set in $ G $. If $ V \subseteq H $ is open in $ H $, then it is also open in $ G $. So there is a finite $ F \subseteq W $ (recall Proposition \ref{finite subset of the RPC set}) with $ W \subseteq VFV $. Thus $ W $ is Roelcke precompact as a subset of $ H $.
	
	Conversely, suppose $ U \subseteq H $ is an open Roelcke precompact neighborhood in $ H $. Then $ U $ is also open in $ G $. Moreover, if $ V \subseteq G $ is an open identity neigbhorhood, there is a finite subset $ F \subseteq H $ with $ U \subseteq (V \cap H)F(V \cap H) \subseteq VFV $.
\end{proof}

Note that this is not true in general for a closed subgroup of a Polish group. For example, both $ \operatorname{Homeo}([0,1]^{\mathbb{N}}) $ and $ \operatorname{Iso}(\mathbb{U}_{1}) $ are universal Polish groups \cite{uspenskii1986universal,uspenskij2008subgroups}, thus each has a topologically isomorphic copy of the other embedded as a closed subgroup. However, $ \operatorname{Iso}(\mathbb{U}_{1}) $ is a Roelcke precompact group \cite{uspenskij2008subgroups,rosendal2009topological}, while $ \operatorname{Homeo}([0,1]^{\mathbb{N}}) $ is coarsely bounded but not Roelcke precompact \cite{rosendal2009topological,rosendal2013global}, so cannot be locally Roelcke precompact by Corollary \ref{RPC iff LRPC and OB}.

\begin{thm}
	If $ G $ is Polish locally Roelcke precompact and $ N $ is a closed normal subgroup, then the quotient group, $ G/N $, is locally Roelcke precompact.
\end{thm}

\begin{proof}
	Let $ \pi\colon G \to G/N $ be the (continuous, open) quotient map. Suppose $ U \subseteq G $ is an open Roelcke precompact identity neighborhood. Then $ \pi[U] $ is open in $ G/N $, and if $ V $ is open in $ G/N $, there is a finite $ F \subseteq G $ such that $ U \subseteq \pi^{-1}[V] F \pi^{-1}[V] $, and so
	\[ \pi[U] \subseteq \pi[\pi^{-1}[V] F \pi^{-1}[V]] = (\pi\pi^{-1}[V]) (\pi[F]) (\pi\pi^{-1}[V]) = V \pi[F] V. \]
\end{proof}

\section{Semigroup operations extending multiplication} \label{sec semigroup operation}

The results of this section are inspired by those of section 5 of \cite{ben2016weakly}, that for $ G $ the (necessarily Roelcke precompact and Polish) automorphism group of a separable metric structure, every Roelcke uniformly continuous function on $ G $ is weakly almost periodic if and only if the theory of the structure is stable. It was pointed out to the author by C. Rosendal that this is equivalent to multiplication on $ G $ extending to a separately continuous semigroup operation its Roelcke completion. This can be seen directly, much in the manner of the proof of Theorem \ref{thm loc comp semitopological semigroup} below, or abstractly---as for Roelcke precompact groups, the Roelcke completion coincides with the compactification associated to the Roelcke uniformly continuous functions and factors onto the compactification associated to the weakly almost periodic functions, which is the universal semi-topological semigroup compactification of $ G $. Therefore, the Roelcke completion and WAP-compactification coincides when the corresponding function algebras do. The authors of \cite{ben2016weakly} also offer a model theoretic interpretation of this semigroup structure.

In general \cite[Chapter 10]{roelcke1981uniform}, for $ X $ the Roelcke completion of $ G $, there is a maximal subset $ M \subseteq X \times X $ that supports a jointly continuous operation extending multiplication $ G \times G \to X $. There are natural embeddings of the left and right completions, $ \widehat{G}^{\mathscr{L}} $ and $ \widehat{G}^{\mathscr{R}} $, of $ G $ into $ X $ making all diagrams commute, which satisfy
\[ \widehat{G}^{\mathscr{L}} \times \widehat{G}^{\mathscr{L}}, \ \widehat{G}^{\mathscr{R}} \times \widehat{G}^{\mathscr{R}}, \ \widehat{G}^{\mathscr{R}} \times X, \ X \times \widehat{G}^{\mathscr{L}} \subseteq M. \]
Thus multiplication on $ G $ extends to the above sets, giving $ \widehat{G}^{\mathscr{L}} $ and $ \widehat{G}^{\mathscr{R}} $ the structure of topological semigroups, with jointly continuous actions $ \widehat{G}^{\mathscr{R}} \curvearrowright X $ and $ X \curvearrowleft \widehat{G}^{\mathscr{L}} $. However, multiplication extends to a jointly continuous operation on all of $ X \times X $ if and only if the Roelcke completion of $ G $ coincides with the two-sided completion; in the setting of Polish groups this occurs if and only if $ G = X $, or equivalently if $ G $ is Weil complete, which for locally Roelcke precompact groups means that $ G $ is locally compact by Corollary \ref{cor loc comp iff cli and lrpc}.

So the best that can be hoped for in a general locally Roelcke precompact group is the structure of a \emph{semi-topological} semigroup (i.e., where multiplication is separately continuous). The methods of \cite{ben2016weakly} completely describe this for Roelcke precompact Polish groups. Here we investigate the corresponding situation for the locally Roelcke precompact ones. We first see that in such situations, the operation can be further extended to the one-point compactification of $ X $.

\begin{prop} \label{opc semitopological semigroup}
	Suppose $ G $ is locally Roelcke precompact but not Roelcke precompact, that $ X $ is its Roelcke completion and $ X_{*} $ is the one-point compactification of $ X $. If multiplication in $ G $ extends to a separately continuous semigroup operation on $ X $ then it further extends to a separately continuous semigroup operation on $ X_{*} $ with $ \infty $ defined to be a zero element.
\end{prop}

\begin{proof}
	The resulting structure on $ X_{*} $ is always a semigroup; checking that the extended multiplication is separately continuous amounts to checking that this is still true at $ \infty $, that is, for every $ s \in X $ and compact $ K \subseteq X $, the sets $ \{x \in X \mid sx \in K \} $ and $ \{x \in X \mid xs \in K \} $ are compact \cite[Example 1.3.3(d)]{berglund1989analysis}. By Theorem \ref{LRPC iff completion is locally compact} there is a sufficiently large $ d_{\wedge} $-ball, $ B $, around $ 1_{G} $ in $ G $ whose closure in $ X $ contains $ K $. Let $ A $ be any open set in $ G $ whose closure in $ X $ contains $ s $ in the interior, e.g. a $ d_{\wedge} $-ball of radius $ 2 $ around any element of $ G $ of distance less than $ 1 $ from $ s $. Suppose $ sx \in K $. Pick a sequence $ (g_{n})_{n} $ from $ G $ with $ g_{n} \to s $. By separate continuity, $ g_{n}x \to sx $. So there is $ N \in \mathbb{N} $ so that $ g_{N} \in A $ and $ g_{N}x \in \operatorname{int}(\overline{B}) \subseteq X $. Letting $ (h_{m})_{m} $ be a sequence in $ B $ tending to $ g_{N}x $, by separate continuity
	\[ x = g_{N}^{-1}g_{N}x = g_{N}^{-1} \lim_{m} h_{m} = \lim_{m}g_{N}^{-1}h_{m} \in \overline{A^{-1}B}. \]
	Since $ A $ and $ B $ are Roelcke precompact in $ G $ so is $ A^{-1}B $ by Proposition \ref{R closed under products in LRPC}, and as $ x $ was arbitrary, $ \{x \in X \mid sx \in K \} $ is a subset of the compact set $ \overline{A^{-1}B} $. Being the preimage of $ K $ under the map $ x \mapsto sx $, it is also closed, and therefore compact.
\end{proof}

Note that there is a clear converse to this result. In the above proposition, the zero element, $ \infty $, of $ X_{*} $ is \emph{removable} in the sense that $ ab = \infty $ if and only if $ a = \infty $ or $ b = \infty $. And clearly if $ S $ is any semigroup with such an element, $ 0 $, then the subset $ S \setminus \{0\} $ is a sub-semigroup. So multiplication in $ G $ extends to a semi-topological semigroup operation on $ X $ if and only if it extends to a semi-topological semigroup operation on $ X_{*} $ with a removable zero at $ \infty $.

Recall that a function is uniformly continuous with respect to the Roelcke uniformity if and only if it is uniformly continuous in both the left- and right-uniformities. Let $ C_{0}'(G) $ denote those functions on $ G $ that vanish off of the ideal of Roelcke precompact sets:
\[ C_{0}'(G) = \left\{f \mid \forall \varepsilon > 0 \left\{g \in G \mid \left| f(g) \right| > \varepsilon \right\} \text{ is a Roelcke precompact subset of } G \right\}. \]
The idea here is that the Roelcke uniformly continuous functions in $ C_{0}'(G) $ are precisely the restrictions of those in $ C_{0}(X) $.

\begin{thm} \label{thm loc comp semitopological semigroup}
	Let $ G $ be a locally Roelcke precompact Polish group. The group multiplication in $ G $ extends to a separately continous semigroup operation on the Roelcke completion, $ X $, if and only if every Roelcke uniformly continuous function in $ C_{0}'(G) $ is weakly almost periodic.
\end{thm}

\begin{proof}
	As mentioned above, this follows from \cite{ben2016weakly} for Roelcke precompact groups. And every locally compact Polish group is its own Roelcke completion, while in that case every function in $ C_{0}'(G) = C_{0}(G) $ is weakly almost periodic \cite[Corollary 4.2.13]{berglund1989analysis}. So assume that $ G $ is not in either of those classes.
	
	Suppose multiplication extends to a separately continuous semigroup operation on $ X $ and let $ f $ be a Roelcke uniformly continuous function in $ C'_{0}(G) $. Then by Proposition \ref{opc semitopological semigroup}, multiplication extends further to the one point compactification, $ X_{*} $. As $ f $ is Roelcke uniformly continuous on $ G $, it extends to a uniformly continuous function on $ X $ and since, for all $ \varepsilon > 0 $, the set $ \left\{x \in X \mid \left| f(x) \right| > \varepsilon \right\} $ is open, its closure is equal to the closure of $ \left\{g \in G \mid \left| f(g) \right| > \varepsilon \right\} $, and so is compact. Therefore $ f \in C_{0}(X) $, and so extends to continuous function on $ X_{*} $. Being a compact semitopological semigroup, every continuous function on $ X_{*} $ is weakly almost periodic and, moreover, the restriction of every weakly almost periodic function on $ X_{*} $ is a weakly almost periodic function on $ G $ \cite[Corollary 4.2.9 and Theorem 4.2.10]{berglund1989analysis}, so $ f \colon G \to \mathbb{C} $ is weakly almost periodic.
	
	Conversely, suppose that every Roelcke uniformly continuous function tending to $ 0 $ off the ideal of Roelcke precompact subsets is weakly almost periodic. Recall that left (resp. right) multiplication on $ G $ extends to a jointly continuous left (resp. right) action on $ X $, and suppose $ x,y \in X $ and $ (g_{n}),(h_{m}) $ are sequences in $ G $ with $ g_{n} \to x $ and $ h_{m} \to y $. Then both sequences are $ d_{\wedge} $-bounded, and therefore there are $ d_{\wedge} $-balls $ A $ and $ B $ around $ 1_{G} $ of sufficient radius so that all terms $ g_{n}h_{m} $, $ g_{n}y $ and $ xh_{m} $ lie in the interior of the compact set $ \overline{AB} \subseteq X $. Then for any nonprincipal ultrafilters $ \mathcal{U}, \mathcal{V} $ on $ \mathbb{N} $, the limits
	\[ \lim\limits_{n \to \mathcal{U}} g_{n}y \text{ and } \lim\limits_{m \to \mathcal{V}} xh_{m} \text{ and } \lim\limits_{n \to \mathcal{U}} \lim\limits_{m \to \mathcal{V}} g_{n}h_{m} \text{ and } \lim\limits_{m \to \mathcal{V}}\lim\limits_{n \to \mathcal{U}} g_{n}h_{m} \]
	all exist, and moreover by continuity of the actions,
	\[ \lim\limits_{n \to \mathcal{U}} g_{n}y = \lim\limits_{n \to \mathcal{U}} g_{n}(\lim\limits_{m \to \mathcal{V}} h_{m}) = \lim\limits_{n \to \mathcal{U}} \lim\limits_{m \to \mathcal{V}} g_{n} h_{m} \]
	and likewise $ \lim_{m \to \mathcal{V}} x h_{m} = \lim_{m \to \mathcal{V}} \lim_{n \to \mathcal{U}} g_{n}h_{m} $. The above is true for any locally Roelcke precompact Polish $ G $, but under the above weak almost periodicity assumption, we also see that $ \lim_{n \to \mathcal{U}} \lim_{m \to \mathcal{V}} g_{n}h_{m} = \lim_{m \to \mathcal{V}} \lim_{n \to \mathcal{U}} g_{n}h_{m} $ for all $ x,y,(g_{n}),(h_{m}), \mathcal{U} $ and $ \mathcal{V} $ as above. For suppose $ a, b \in X $ with 
	\[ a = \lim_{n \to \mathcal{U}} \lim_{m \to \mathcal{V}} g_{n}h_{m} \neq \lim_{m \to \mathcal{V}} \lim_{n \to \mathcal{U}} g_{n}h_{m} = b  \]
	and let $ f \colon G \to \mathbb{C} $ be the uniformly continuous function, $ f(g) = d_{\wedge}(a,b) - \min \{d_{\wedge}(a,b),d_{\wedge}(a,g)\} $, (where the distances are computed in $ X $, though note that the unique extension to a uniformly continuous function on $ X $ is given by the same formula). Then $ |f(g)| > 0 $ if and only if $ g $ is in the trace in $ G $ of the ball around $ a $ of radius $ d_{\wedge}(a,b) $, so $ f \in C_{0}'(G) $.
	
	So $ f $ is Roelcke uniformly continuous and in $ C_{0}'(G) $, yet
	\begin{align*}
	\lim\limits_{n \to \mathcal{U}}\lim\limits_{m \to \mathcal{V}} f(g_{n} h_{m}) &= f \left(\lim\limits_{n \to \mathcal{U}}\lim\limits_{m \to \mathcal{V}} g_{n} h_{m} \right) = f(a) \neq f(b) = \\
	&= f \left( \lim\limits_{m \to \mathcal{V}}\lim\limits_{n \to \mathcal{U}} g_{n} h_{m} \right) = \lim\limits_{m \to \mathcal{V}}\lim\limits_{n \to \mathcal{U}} f(g_{n} h_{m})
	\end{align*}
	and so $ f $ is not weakly almost periodic, a contradiction.
	
	Thus for every $ g_{n} \to x $ and $ h_{m} \to y $ and nonprincipal $ \mathcal{U} $ and $ \mathcal{V} $, the values $ \lim_{n \to \mathcal{U}} \lim_{m \to \mathcal{V}} g_{n}h_{m} $ and $ \lim_{m \to \mathcal{V}} \lim_{n \to \mathcal{U}} g_{n}h_{m} $ agree. But in fact, this is then also independent of both the choice of sequences and of ultrafilters. For if also  $ g_{n}' \to x $ and $ \mathcal{U}' $ is some nonprincipal ultrafilter on $ \mathbb{N} $,
	\[ \lim\limits_{m \to \mathcal{V}}\lim\limits_{n \to \mathcal{U}} g_{n}h_{m} = \lim\limits_{m \to \mathcal{V}} xh_{m} = \lim\limits_{m \to \mathcal{V}}\lim\limits_{n \to \mathcal{U'}} g_{n}'h_{m} \]
	and likewise $ \lim_{n \to \mathcal{U}} \lim_{m \to \mathcal{V}} g_{n}h_{m} = \lim_{n \to \mathcal{U}} \lim_{m \to \mathcal{V'}} g_{n}h_{m}' $ for any other $ h_{m} \to y $ and $ \mathcal{V}' $.
	
	Therefore we may define an operation $ X \times X \to X $, where $ x \cdot y $ is the common value of $ \lim_{n \to \mathcal{U}} \lim_{m \to \mathcal{V}} g_{n}h_{m} $ for any appropriate $ (g_{n}), (h_{m}), \mathcal{U} $ and $ \mathcal{V} $. If $ g_{n} \to x $, $ h_{m} \to y $ and $ k_{i} \to z $ then
	\begin{align*}
	(x \cdot y) \cdot z &= \lim_{i \to \mathcal{W}} \left(\lim_{n \to \mathcal{U}} g_{n}y\right)k_{i} = \lim_{i \to \mathcal{W}} \lim_{n \to \mathcal{U}} (g_{n}y)k_{i} = \lim_{i \to \mathcal{W}} \lim_{n \to \mathcal{U}} g_{n}(yk_{i}) \\
	&= \lim_{n \to \mathcal{U}}\lim_{i \to \mathcal{W}} g_{n}(yk_{i}) = \lim_{n \to \mathcal{U}} g_{n} \left(\lim_{i \to \mathcal{W}} yk_{i}\right) = x \cdot (y \cdot z)
	\end{align*}
	and so $ X $ with this operation is a semigroup. The constant sequences show that this operation extends multiplication in $ G $, and to see that multiplication is separately continuous, note that a function, $ f \colon X \to X $, on the complete metric space, $ (X,d_{\wedge}) $, is continuous if and only if for every $ x \in X $ and $ (g_{n}) \subseteq G $ with $ g_{n} \to x $, $ f(g_{n}) \to f(x) $, and that the latter occurs if and only if $ \lim_{ n \to \mathcal{U}}f(g_{n}) = f(x) $ for every non-principal ultrafilter, $ \mathcal{U} $. Then apply these observations to the functions $ x \mapsto xy $ and $ x \mapsto yx $.
\end{proof}

Note however that unlike the situation for Roelcke precompact Polish groups, where the Roelcke compactification coincides with the WAP compactification for such groups supporting a separately continuous extension of multiplication, we do not obtain a description of the WAP compactification in this setting. For example, as a consequence of \cite[Corollary 2.2]{ferri2004note}, the WAP compactification of a locally compact Polish SIN group must have cardinality $ 2^{2^{\aleph_{0}}} $.

\nocite{roe2003lectures}

\bibliographystyle{abbrv}
\bibliography{LRPClast}

\end{document}